\documentclass[12pt,reqno]{amsart}
\usepackage{hyperref}
\usepackage{graphicx,epsfig}
\usepackage[comma,sort&compress]{natbib}
\usepackage{fancyhdr,latexsym,amsmath,amsfonts,amssymb,amsbsy,amsthm,url}

\usepackage{euscript}
\allowdisplaybreaks
\newtheorem{theorem}{Theorem}[section]
\newtheorem{lemma}{Lemma}[section]
\newtheorem{proposition}{Proposition}[section]

\theoremstyle{definition}

\theoremstyle{remark}

\numberwithin{equation}{section}

\newlength{\myfboxsep}
\newlength{\mywidth}
\DeclareMathOperator{\Tr}{Tr}

\newcommand{\comment}[1]{}

\newcommand{\bw}{\mathcal W}

\begin{document}
\bibliographystyle{abbrvnat}
\title[Block Toeplitz matrices]{Limiting Spectral Distribution of Block Matrices
with  Toeplitz Block Structure}
\author[R. Basu]{Riddhipratim Basu}
\address{Department of Statistics,
 University of California, Berkeley}
\email{riddhipratim@stat.berkeley.edu}
\author[A.Bose]{Arup Bose}
\thanks{ A. Bose's research
supported by J.C.Bose National Fellowship, Dept. of Science and Technology,
Govt. of India.}
\address{Statistics and Mathematics Unit,  Indian
Statistical Institute 203 B.T.~Road Kolkata 700108 India}
\email{bosearu@gmail.com}
\author[S. Ganguly]{Shirshendu Ganguly}
\address{Department of Mathematics, University of Washington,
Seattle}
\email{sganguly@math.washington.edu}
\author[R. S. Hazra]{Rajat Subhra Hazra}
\address{Institut f\"ur Mathematik\\
Universit\"at  Z\"urich\\
Winterthurerstrasse 190\\
CH-8057, Z\"urich } \email{rajatmaths@gmail.com}

\date{November 04, 2011}

\begin{center}
\end{center}

\begin{abstract}
We study two specific symmetric random block Toeplitz (of dimension
$k \times k$) matrices:  where the blocks (of size $n \times n$) are
(i) matrices with i.i.d. entries
 and (ii) asymmetric Toeplitz matrices. Under suitable
 assumptions on the entries, their  limiting spectral distributions (LSDs) exist (after scaling by $\sqrt{nk}$)
when (a) $k$ is fixed and $n \to
\infty$ (b) $n$ is fixed and $k\rightarrow \infty$
 (c)  $n$ and $k$ go to $\infty$ simultaneously. Further the LSD's
 obtained in (a) and (b) coincide with those in (c) when $n$ or
 respectively $k$ tends to infinity. This limit in (c) is the
 semicircle law in case (i). In Case (ii) the limit is related to
the limit of the random symmetric Toepiltz matrix as obtained by
 \cite{bry} and \cite{hammond:miller}.
 \end{abstract}
\maketitle
\section{Introduction}
Limiting Spectral Distribution (LSD) of block matrices with large
random blocks has been studied in the literature. Under certain
moment assumptions, \cite{Oraby} proved a general existence
theorem for LSD of such matrices with large Wigner blocks and
finite symmetric block structure. \cite{bana:bose} later extended
these results to relax the moment assumptions and also to include
Wigner type blocks. They also showed the limit to be the
semicircular law when both the size of a block and the size of the
block structure is increased in a suitable way.

Our focus in this article is on matrices with the Toeplitz block
structure. Matrices of the form $B_k=((x_{i-j}))_{1\leq i, j \leq
k}$ where $x_i$ is some sequence of numbers are called
\textit{Toeplitz matrices}. If the elements $x_i$ are themselves $n
\times n$ matrices, say $A_{n,i}\equiv A_i$, with $A_i=A_{-i}^T$ for
all $i$ then the corresponding \text{symmetric} matrices (say $B_k
\star \{A_{n,i}\}$) are \textit{block matrices with asymmetric
Toeplitz block structure}. Visually,
\begin{equation}\label{def: toep4}
         B_k \star \{A_{n,i}\}= \left[ \begin{array} {cccccc}
            A_{0} & A_{-1} & A_{-2} & \ldots & A_{-(k-2)} & A_{-(k-1)} \\
            A_{1} & A_{0} & A_{-1} & \ldots & A_{-(k-3)} & A_{-(k-2)} \\
            A_{2} & A_{1} & A_{0} & \ldots & A_{-(k-4)} & A_{-(k-3)} \\
                  &       &       & \vdots &       &   \\
            A_{k-1} & A_{k-2} & A_{k-3} & \ldots & A_{1} & A_{0}
            \end{array} \right].
\end{equation}

Block Toeplitz matrices arise in many aspects of mathematics,
physics and technology.  \cite{gaza} researched the asymptotic
behaviour of eigenvalue distribution for deterministic block Toeplitz matrices.
 \cite{oraby2}  proved the
existence of the LSD for random block Toeplitz matrices with
self-adjoint $\{A_{n,i}\}$ as $n\rightarrow \infty$. They assumed
that the entries are complex Gaussian. The proofs use an
operator-valued free probability approach and Wick's formula for
moments of Gaussian variables. \cite{block} study the LSD of these
block Toeplitz matrices ($B_k \star \{A_{n,i}\}$) when $A_{n,i}$
have i.i.d. entries (with uniformly bounded moments of all order)
 as
$k\rightarrow \infty $ and show the existence of the LSD. They also
show that when in turn $n \to \infty$, the above limit converges to
the semicircle law.
They use the moment method and explicit calculations based on the
eigenvalue trace formula.
%
%

We deal with two specific symmetric block matrices where the blocks
are random Toeplitz matrices with $A_i=A_{-i}^T$ for all $i$ and are
otherwise independent:\vskip5pt

Case  (i) The entries of $A_i$ are i.i.d. (as in
\cite{block}).\vskip5pt

Case  (ii) $A_i$ are (asymmetric) Toeplitz matrices.
This symmetric matrix with asymmetric Toeplitz block structure and
individual Toeplitz blocks does not seem to have been studied in
literature. \vskip5pt

We show that in both cases, their  LSD's exist (under finiteness of
second moments) when\vskip5pt

(a) $k$ is fixed and $n \to \infty$

(b) $n$ is fixed and $k\rightarrow \infty$

 (c)  $n$ and $k$ go to $\infty$ simultaneously.\vskip5pt

 Further the LSD's
 obtained in (a) and (b) coincide with those in (c) when $n$ or
 respectively $k$ tends to infinity.
 This limit in (c) is the
 semicircle law in Case (i). In particular this establishes results of
 \cite{block} under weaker moment conditions
 when the blocks are i.i.d. In Case (ii) the limit is related to
the limit of the random symmetric Toepiltz matrix as obtained by
 \cite{bry} and \cite{hammond:miller}. Our study shows that one can also accommodate other kinds of
patterned matrices in place of the i.i.d.\ or Toeplitz matrices
$A_i$.

\cite{block} also studied the Hankel block structure where the
individual blocks are i.i.d.\ matrices and the symmetry of the whole
matrix was assumed. It would be interesting to  study Hankel block
structures where individual blocks are either Toeplitz or Hankel.
 Note that the Hankel matrix is a symmetric matrix 
and hence the  block structure is symmetric. This is an additional
feature which is absent in the two block matrices that we deal with
in this article. We shall deal with such symmetric block matrices
separately in~\cite{rbsh:2011}.
This symmetry will be exploited to establish additional results and
for the more general situation
where there is a general patterned block structure and
where each individual block is another patterned matrix.

Incidentally, block matrices with asymmetric Toeplitz block
structures are specific block versions of  patterned matrices. The
problem of showing existence of LSD of random patterned matrices
using the moment method was studied in the general framework by
\cite{bosesen} who built upon the so called volume method ideas that
were first propounded by  \cite{bry} in the context of Toeplitz and
Hankel matrices. Our proofs are  based on these ideas and results
and we  also draw generously from the developments in
\cite{bana:bose}.
 \vskip5pt

In Section \ref{section:prelim} we recollect a few useful concepts
from \cite{bosesen} and introduce our model. In Section
\ref{section:results} we state and prove our results.

%


\section{Preliminaries}\label{section:prelim}
\subsection{Background material}
We first present in brief some of the material from \cite{bosesen}
that we need.
A patterned matrix is defined through a \emph{link function}. Let
$d$ be a positive integer. Let $L_n$ be functions
$L_n:\{1,2,..,n\}^2 \rightarrow Z_{\geq}^d, ~n \geq 1.$ Here
$Z_{\geq}^d$ denotes all $d$-tuples of non-negative integers ($d=1$
or $2$). We write $L$ for $L_n$ and write $\mathbb{N}^2$ as the
common domain of $\{ L_n \}$.

Then the  sequence of patterned matrices $\{A_n\}$ of order $n
\times n$ with link function $L$ is defined as
$$A_n\equiv((a_{i,j}))=((x_{L_n(i,j)}))$$
where $\{x_{i,j}\}$ or  $\{x_{i}\}$ as the case may be, is called
the \textit{input sequence}.  $L_W$ and $L_T$ will denote the link
functions respectively for the Wigner and symmetric Toeplitz
matrices, so that
$$L_W(i,j)= (\min (i,j),\max (i,j)), \ \ L_T (i,j)=|i-j|.$$

 The $h$-th moment of the {\it
empirical spectral distribution (ESD)} of $n^{-1/2}A_n$ equals
\begin{equation*}\label{eq:tracepi}
\frac{1}{n}\sum_i
(\frac{\lambda_i}{\sqrt{n}})^h=\frac{1}{n}\Tr\left(\frac{A_n}{\sqrt n}\right)^h =
\frac{1}{n^{1+h/2}}\sum_{1 \le i_1,i_2,\ldots,i_h \le n} x_{L(i_1,
i_2)} x_{L(i_2, i_3)} \cdots x_{L(i_{h-1}, i_h)} x_{L(i_h, i_1)}.
\end{equation*}
Any function $\pi:\{0,1,2,\cdots,h\} \rightarrow \{1,2,\cdots,n\}$ \
is
a {\it circuit of length h} if $\pi(0) = \pi(h).$
A circuit depends on  $h$ and $n$ but we suppress this dependence. A
circuit $\pi$ is said to be \emph{matched}).
if given any $i$, there is at least one $j\ne i$ such that
$L(\pi(i-1),\pi(i)) = L(\pi(j-1),\pi(j)).$\vskip5pt



Two circuits $\pi_1$ and $\pi_2$ (of same length) are
\textit{equivalent} if
$\big\{L(\pi_1(i-1), \pi_1(i)) =
L(\pi_1(j-1), \pi_1(j))$ $\Leftrightarrow L(\pi_2(i-1), \pi_2(i))
= L(\pi_2(j-1), \pi(j))\big\}. $ This defines an equivalence
relation.
\textit{Equivalence classes} are identified with partitions of
$\{1,2,\cdots,h\}$. Partitions will be labelled by  \textit{words}
$w$ of length $l(w)=h$ of letters where the first occurrence of each
letter  is in alphabetical order. For example, if $h = 5$, then the
partition $\{ \{1,3,5\}, \{2,4\}\} $ is represented by the word
$ababa$. %
 Let $w[i]$ denote the $i$-th entry of
$w$.  The equivalence class corresponding to $w$ will be denoted by
$$ \Pi_{L,n}(w) = \{\pi: w[i] = w[j] \Leftrightarrow L(\pi(i-1),\pi(i)) =
L(\pi(j-1),\pi(j))\}.$$


A word is \textit{pair-matched} if every letter occurs exactly two
times. The set of pair-matched words of length $k$ will be denoted
by $\bw_{k}(2)$. For any set $G$, let $\# G$ denote the number of
elements in $G$.

Note that $\#\bw_{k}(2)=\frac{(2k)!}{2^k k!}$.
A word $w \in \bw_{k}(2)$ is \textit{Catalan} if it has a double letter and removing double
letters successively leads to the empty word. \vskip5pt

 Define for any (matched) word $w$,
$$\Pi_{L,n}^*(w) = \{\pi: w[i] = w[j] \Rightarrow L(\pi(i-1),\pi(i)) =
L(\pi(j-1), \pi(j)) \}\  \supseteqq \Pi_{L, n}(w).$$
$\Pi_{L, n}^*(w)$  is often equivalent to $\Pi_{L, n}(w)$ for
asymptotic considerations, but is easier to work with.\vskip5pt

Any $\pi(i)$ is 
a \textit{vertex} and it is
\emph{generating} if either $i=0$ or $w[i]$ is the
position of the \emph{first} occurrence of  a letter. For example,
if $w = abbcab$ then $\pi(0),\pi(1),\pi(2),\pi(4)$ are generating
vertices.\vskip5pt

%

For any word of length $2t$, let (whenever the limit exists),
$$p(w)=\lim_{n\rightarrow \infty}\frac{\#\Pi_{L, n}^*(w)}{n^{1+t}}.$$
It is known that with appropriate input sequences, $p(w)$ exists for
many symmetric patterned matrices, including the Wigner and the
symmetric Toeplitz matrices and $\displaystyle{\sum_{w\in \bw_k(2)}
p(w)}$ is the $2k$-th moment of the LSD. We shall denote by $p_W(w)$
and $p_T(w)$ the value of $p(w)$ for these two matrices
respectively.

\subsection{Our model}
A \textit{finite block structure} $B_k=B_k(a_0,a_1,...)$ is a
$k\times k$ patterned matrix with link function $L_1$ and input
sequence $\{a_i\}_{i\in \mathbb{Z}}$. $B_k$ is called symmetric if
$L_1(i,j)=L_1(j,i)$. Let $\{A_{n,0}, A_{n,1},...\}_{n\times n}$ be a
sequence of independent patterned matrices with link functions
$L_{2,0}, L_{2,1},...$ and respective independent input sequences
$\{x_{0,k},x_{1,k},...\}_{k\geq 1}$. Then the $(i,j)-th$ entry of
$A_{n,k}$ is given by $x_{k,L_{2,k}(i,j)}$. The $kn\times kn$ block
matrix
$B_k \star A_{n , i}$ is
defined by replacing $(i,j)-th$ entry of $B_k$ by $A_{n,L_1(i,j)}$,
i.e. $B_k \star A_{n,i}=B_k(A_{n,0},A_{n,1},...).$ Let
$$m(i,j)=L_1\left(\left[\frac{i-1}{n}\right]+1 ,\left[\frac{j-1}{n}\right]+1\right).$$ Thus  $B_k \star A_{n , i}$ is a
$nk \times nk$ matrix with link function $L$ defined as
$$L(i,j)= (m(i,j), L_{2,m(i,j)}\left((i-1)\mod ~n +1, (j-1)\mod~ n
+1)\right).$$ Hence all the concepts introduced in the previous
subsection remain valid. In this article we shall deal with the
following two sequences of block matrices. It may be noted that all of these
are symmetric matrices, even though most of the blocks are asymmetric.\vskip5pt

\textbf{Toeplitz block matrix with i.i.d. blocks}
 ($TBI_{k,n}$): Here  $L_1(i,j) = i-j$ is the link function for the block structure $B_k$,  $A_{0,n}$ is a Wigner matrix
with link function $L_W$ and  for $i>  0$, $A_{n , i}$ are  i.i.d.
matrices  with the common  link function $L_2(k,l)=(k,l)$ for every
$k,l$. For $i<0$, $A_{n , i}=A_{n, -i}^T$. We denote this block
matrix by $TBI_{k,n}$.\vskip5pt

\textbf{Toeplitz block matrix with Toeplitz blocks}
($TBT_{k,n}$):
 Here $B_k$ is as above, $A_{0,n}$ is a symmetric Toeplitz
matrix with  $L_{2,0}(k,l)=L_T (k,l)$ and for $i>0$, $A_{n , i}$ is
an asymmetric Toeplitz matrix with link function $L_2(k,l)=k-l$. For
$i<0$, $A_{n , i}=A_{n, -i}^T$. We denote this block matrix by
$TBT_{k,n}$.\vskip5pt

We shall make the following assumption on the input sequence:
\vskip5pt

 \textit{Assumption A}: The input sequence is independent
with mean 0 and variance 1 and either uniformly bounded or
identically distributed or with uniformly bounded moments of all
order. \vskip5pt

\subsection{Address functions}
The following definitions are taken from  \cite{bana:bose} and
modified for our set-up.\vskip5pt

{\bf Block Address  Function:}
Let us consider any of the
above two  block matrices.
 Let $U_i =\{(i-1)n+1,...,in\},i=1,2,...,k$, be a
partition of ${1,2, \ldots, nk}$. 

For any $w\in \bw_{2t}(2)$,
the {\bf block address} $\pi_b$ of $\pi \in \Pi_{L, nk}^*(w)$ is
defined as:
$$\pi_b(i)=j~~\text{if}~~\pi(i)\in U_j.$$ It denotes which of the $k$ blocks of rows
(or columns) an index belongs to. Note that from the definition of
the block structure, if two elements of $B_k\star A_{n, i}$ are
same, then they must belong to blocks which are either same or
transpose of each other. Hence,
\begin{eqnarray}
w[i]=w[j]&\Rightarrow &
L_1(\pi_b(i-1),\pi_b(i))= \pm L_1(\pi_b(j-1),\pi_b(j))\\
&
\Rightarrow & L_T(\pi_b(i-1),\pi_b(i))= L_T(\pi_b(j-1),\pi_b(j)).
\end{eqnarray} Also $\pi(0)=\pi(2t)\Rightarrow \pi_b(0)=\pi_b(2t)$.
Hence, $\pi_b\in \Pi_{L_T, k}^*(w)$.

This leads to the definition of the {\bf block address function} as:
$\phi_B : \Pi_{L, nk}^*(w)\rightarrow \Pi_{L_T, k}^*(w)$ by:
$$ (\pi(0),....,\pi(2t))\mapsto (\pi_b(0),....,\pi_b(2t)).$$

\textbf{Entry Address function}:  Analogous to the block address we
define entry address as follows. Let $w \in \bw_{2t}(2)$.
The {\bf entry
address} $\pi_e$ of $\pi \in \Pi_{L, nk}^*(w)$ is defined as
$$\pi_e(i)=\pi(i)-(\pi_b(i)-1)n.$$ Clearly then, $1\leq \pi_e(i)\leq
n$. This denotes the address of an entry inside a block. Now, from
the definitions it follows that
$$
w[i]=w[j]\Rightarrow L_{2,m_1}(\pi_e(i-1),\pi_e(i))=
L_{2,m_2}(\pi_e(j-1),\pi_e(j))$$ where
$$m_1=L_1(\pi_b(i-1),\pi_b(i))\ \ \text{and}\ \ m_2=L_1(\pi_b(j-1),\pi_b
(j)).$$ Conversely,  $$L_{2,m_1}(\pi_e(i-1),\pi_e(i))=
L_{2,m_2}(\pi_e(j-1),\pi_e(j)) \Rightarrow m_1=\pm m_2$$ and
hence\vskip5pt

(i)  $(\pi_e(i-1),\pi_e(i))=(\pi_e(j-1),\pi_e(j))$ or
$(\pi_e(i-1),\pi_e(i))=(\pi_e(j),\pi_e(j-1))$ (when we have iid
blocks)\vskip5pt

(ii) $\pi_e(i)-\pi_e(i-1)=\pm (\pi_e(j)-\pi_e(j-1))$ (when we have
asymmetric Toeplitz blocks). \vskip5pt

Also, as before $\pi_e(0)=\pi_e(2t)$ and hence $\pi_e\in \Pi_{L_W,
n}^*(w)$ (when we have iid
blocks) and $\pi_e\in \Pi_{L_T,
n}^*(w)$ (when we have
asymmetric Toeplitz blocks).

This leads to the definition of the {\bf entry address function}
$\phi_A : \Pi_{L, nk}^*(w)\rightarrow \Pi_{L_W, n}^*(w)$ (when we have iid
blocks) or $\phi_A : \Pi_{L, nk}^*(w)\rightarrow \Pi_{L_T, n}^*(w)$ (when we have asymmetric Toeplitz
blocks) as:
$$ (\pi(0),....,\pi(2t))\mapsto (\pi_e(0),....,\pi_e(2t)).$$
Note that $(\pi_b,\pi_e)$ determines $\pi$ uniquely but for any
$\pi_b\in \Pi_{L_T, k}^*(w)$ and $\pi_e\in \Pi_{L_W, n}^*(w)$, the
$\pi$ determined by $\pi_b$ and $\pi_e$ need not be in $\Pi_{L,
nk}^*(w)$.

Now let $w \in \bw_{2t}(2)$.
Let $S$ be the
set of all non-zero generating vertices of $w$. For every $i \in S$,
let us denote by $j_i$ the index such that $w[j_i]$ is the second
occurrence of the letter $w[i]$. Let $i_1,i_2,...,i_t$ be all the
non-zero generating vertices. Let $l=(l_{i_1},...,l_{i_t}) \in
\{-1,1\}^{t}$.\vskip5pt

Let $\Pi_{L_T, k, l}^*(w)$ be the subset of $\Pi_{L_T,k}^*(w)$ such
that,
$$\pi_e(i-1)-\pi_e(i)=l_i(\pi_e(j_i-1)-\pi_e(j_i))~~\forall i \in S.$$
Now clearly, $$\Pi_{L_T, k}^*(w)=\bigcup_{l} \Pi_{L_T, k,l}^*(w)$$
which is {\it not} a disjoint union.

Also, define  $\Pi_{L_W, n,l}^*(w)=\{\pi_e \in \Pi_{L_W, n}^*(w):
w[i]=w[j]\Rightarrow (\pi_e(i-1),\pi_e(i))=(\pi_e(j-1),\pi_e(j))
\quad \text{if} ~l_i=1 \quad\text{and }
\pi_e(i-1),\pi_e(i))=(\pi_e(j),\pi_e(j-1))\quad\text{if} ~l_i=-1
\} $

Clearly,
$$\Pi_{L_W, n}^*(w)=\bigcup_{l} \Pi_{L_W, n,l}^*(w).$$
It is to be noted that $\Pi_{L_W, n,l}^*(w)$'s are closely related
with $R_1$ and $R_2$ constraints as described in \cite{bosesen}.
In fact, it is easy to see that the pair $(i,j_i)$ satisfies an
$R_1$ constraint if $l_i=1$ and it satisfies an $R_2$ constraint
if $l_i=-1$.
\section{Results and proofs}\label{section:results}

\subsection{Results}

We are now ready to state the main results.  \\

\begin{theorem}\label{theorem:wignerblock}
Consider the block matrix $TBI_{k,n}$ where $A_{n,i}$ satisfy
Assumption A.
Then,\vskip5pt

\noindent (i) for fixed $k$, as $n\rightarrow \infty$ LSD of
$\frac{1}{\sqrt{nk}}TBI_{k,n}$ exists w.p. 1, and has all moments
finite, all odd moments 0, and $\forall t>0$, the $2t-th$ moment
$\beta_{2t}$ satisfy
$$\beta_{2t}=\sum_{w \in \bw_{2t}(2),\  w~{\rm catalan}}
\frac{\#\Pi_{L_T, k,l_0}^*(w)}{k^{t+1}}\ \ \text{where}\ \
l_0=(-1,-1,...,-1).$$ \noindent (ii) for fixed $n$, as $k\rightarrow
\infty$, LSD of $\frac{1}{\sqrt{nk}}TBI_{k,n}$ exists w.p. 1, and
has all moments finite, all odd moments 0, and $2t-th$ moments
$\beta_{2t}$ satisfy
$$\beta_{2t}=\sum_{w \in \bw_{2t}(2)}
\frac{\#\Pi_{L_W, n,l_0}^*(w)}{n^{t+1}} p_T(w)\ \ \text{where}\ \
l_0=(-1,-1,...,-1).$$

\noindent (iii) as $n$ and $k$ both go to $\infty$, LSD of
$\frac{1}{\sqrt{nk}}TBI_{k,n}$ exists w.p. 1 and is the
semicircular law.\vskip5pt

\noindent
(iv) in (i) and (ii) if we let $k\rightarrow \infty$ or
$n \rightarrow \infty$ respectively, the LSD converge to that in
(iii).
\end{theorem}

\begin{theorem}\label{theorem:toeplitzblock}
Consider the block matrix $TBT_{k,n}$ where the input sequence
satisfies Assumption B.
Then,\vskip5pt

\noindent (i) for fixed $k$, as $n\rightarrow \infty$ LSD of
$\frac{1}{\sqrt{nk}}TBT_{k,n}$ exists w.p. 1, has all moments
finite, all odd moments 0, and $(2t)-th$  moments $\beta_{2t}$
satisfy
$$\beta_{2t}=\sum_{w \in \bw_{2t}(2)}
\frac{\#\Pi_{L_T, k,l_0}^*(w)}{k^{t+1}} p_T(w)\ \ {\rm where}\quad
l_0=(-1,-1,...,-1).$$ \noindent (ii) for fixed $n$, as $k\rightarrow
\infty$, LSD of $\frac{1}{\sqrt{nk}}TBT_{k,n}$ exists w.p. 1, has all
moments finite, all odd moments 0, and $(2t)-th$  moments
$\beta_{2t}$ satisfy
$$\beta_{2t}=\sum_{w \in
\bw_{2t}(2)}
\frac{\# \Pi_{L_T, n,l_0}^*(w)}{n^{t+1}} p_T(w)\ \ {\rm where}\quad
l_0=(-1,-1,...,-1)$$
\vskip5pt

\noindent (iii) as $n$ and $k$ both go to $\infty$, LSD of
$\frac{1}{\sqrt{nk}}TBT_{k,n}$ exists w.p. 1, is symmetric, has
all moments finite and is determined by the even moments
$$\beta_{2t}= \sum_{w \in
\bw_{2t}(2)}
(p_T(w))^2.$$
\noindent (iv) in (i) and (ii) if we let $k\rightarrow \infty$ or
$n \rightarrow \infty$ respectively, the LSD converge to that in
(iii).
\end{theorem}

\subsection{Proofs}
We first state a Proposition which follows from the results of
\cite{bosesen}
and which
helps to reduce many computational aspects. \\

A link function is said to satisfy \textit{Property B},
if,
\begin{equation}
\Delta(L) = \sup_n \sup_{ t } \sup_{1 \leq k \leq n} \# \{ l: 1 \leq
l \leq n, \ L(k,l) =t\} < \infty.
\end{equation}
\noindent \textit{Assumption B}: Let $k_n=\#\{L_n(i,j):1\leq i,j
\leq n\}$ and $\alpha_n=\max_k\#\{(i,j): L_n(i,j)=k\}$. Then $k_n
\rightarrow \infty$ and $ k_n \alpha_n =O(n^2)$.\vskip5pt


\begin{proposition}\label{prop:review} Suppose $A_n$ is a sequence of $n \times n$ patterned random
matrices with link function $L$ satisfying Assumption B and
\textit{Property B}. \vskip5pt
\begin{enumerate}
\item \label{item1} Suppose for every bounded, mean
zero and variance one i.i.d. input sequence, the LSD exists
almost surely and is non-random.
 Then the same limit continues to hold if the input
sequence satisfies Assumption A.

\item \label{item2} if $w$ is matched but not pair-matched then $p(w)$
exists and is equal to 0. \vskip5pt

\item \label{item3} if for every $t>0$ and for every
$ w\in \bw_{2t}(2)$, $p(w)=\lim_{n\rightarrow \infty}\frac{\#\Pi_{L,
n}^*(w)}{n^{1+t}}$ exists then the LSD of $\frac{A_n}{\sqrt{n}}$
exists and
the $2t$-th moment of the LSD is given by
$$\beta_{2t}=\sum_{w \in \bw_{2t}(2)}
p(w).$$
\item\label{item4} Assumption B and \textit{Property B} are satisfied for the
Toeplitz and the Wigner link functions. Both for the Toeplitz matrix
and the Wigner matrix, $p(w)$ exists for every $t>0$ and every $
w\in \bw_{2t}(2)$. For the Wigner matrix $p_W(w)=0$ iff $w$ is
non-Catalan. For every Catalan word $w$,  $p_W(w)=p_T(w)=1$.
\end{enumerate}
\end{proposition}%

Before we prove Theorems~\ref{theorem:wignerblock} and~\ref{theorem:toeplitzblock} we state a lemma
which quantifies some of the properties of Toeplitz and Wigner limit and these  were proved in~\cite{bry}
and~\cite{bosesen}. See also Theorem 5.1 of \cite{icm}.

\begin{lemma}\label{result2}
 Let $w\in \bw_{2t}(2)$
and $l_0$ be the $k$-tuple $(-1,-1, \ldots,-1)$.\vskip5pt

(1)  If $l\neq l_0$ , then $\lim_{k\rightarrow
\infty}\frac{1}{k^{t+1}}\#\Pi_{L_T, k,l}^*(w)=0.$\vskip5pt

(2)  $\lim_{k\rightarrow \infty}\frac{1}{k^{t+1}}\#\Pi_{L_T,
k,l_0}^*(w)=\lim_{k\rightarrow \infty}\frac{1}{k^{t+1}}\#\Pi_{L_T,
k}^*(w)=p_T(w)$.\vskip5pt

(3)   If $w$ is Catalan, then $\forall~ l\neq l_0,$
$$\lim_{n\rightarrow \infty}\frac{1}{n^{t+1}}\#\Pi_{L_W, n,l}^*(w)=0
\text{ and }\lim_{n\rightarrow \infty}\frac{1}{n^{t+1}}\#\Pi_{L_W,
n,l_0}^*(w)=1.$$\vskip5pt
\end{lemma}

\subsection{Proof of Theorem \ref{theorem:wignerblock}}
We begin by proving a few lemma involving the block  and entry
address functions.
\begin{lemma}\label{lemma:wigner}
Let $w \in \bw_{2t}(2)$.
\begin{enumerate}
\item \label{toeplitz}
Let $\pi_b\in \Pi_{L_T, k}^*(w)$ and $l^1,l^2,...,l^m$ be all the
distinct values of $l$ such that, $\pi_b\in \Pi_{L_T, k,l}^*(w)$.
Then,
$$\#\phi_B^{-1}(\pi_b)=\#\bigcup_{l\in\{l^1,...,l^m\}}\Pi_{L_W, n,l}^*(w).$$
\item \label{wigner}
Let $\pi_e\in \Pi_{L_W, n}^*(w)$ and $l^1,l^2,...,l^m$ be all the
distinct values of $l$ such that, $\pi_e\in \Pi_{L_W, n,l}^*(w)$.
Then,
$$\#\phi_A^{-1}(\pi_e)=\# \bigcup_{l\in\{l^1,...,l^m\}}\Pi_{L_T, k,l}^*(w).$$
\end{enumerate}
\end{lemma}
\begin{proof}[Proof of Lemma~\ref{lemma:wigner}]
(1) Let $\pi \in \phi_B^{-1}(\pi_b)$. Then we claim that
$\pi_e=\phi_A(\pi) \in \bigcup_{l\in\{l^1,...,l^m\}}\Pi_{L_W,
n,l}^*(w)$. To prove this, let $i_1,i_2,..,i_t$ be all the
non-zero generating vertices for $w$. Since $\pi_e\in \Pi_{L_W,
n}^*(w)$, then either,
$(\pi_e(i-1),\pi_e(i))=(\pi_e(j_i-1),\pi_e(j_i))$ or
$(\pi_e(i-1),\pi_e(i))=(\pi_e(j_i),\pi_e(j_i-1))$. For $i\in
\{i_1,i_2,..,i_t\} $, define
\begin{equation*}
m_i =\begin{cases}1 &\quad\text{if } (\pi_e(i-1),\pi_e(i))=(\pi_e(j_i-1),\pi_e(j_i))\\
      -1&\quad\text{if } (\pi_e(i-1),\pi_e(i))=(\pi_e(j_i),\pi_e(j_i-1)).
     \end{cases}
\end{equation*}
Define $m_i$ to be any of 1 or -1 if both the conditions are
satisfied simultaneously (this happens when
$\pi_e(i-1)=\pi_e(i)$). As $\pi \in \Pi_{L, nk}^*(w)$ and
$\phi_A(\pi)=\pi_e, \phi_B(\pi)=\pi_b$, $m_i=1 \Rightarrow
\pi_b(i-1)-\pi_b(i)=\pi_b(j_i-1)-\pi_b(j_i)$, and $m_i=-1
\Rightarrow \pi_b(i-1)-\pi_b(i)=-(\pi_b(j_i-1)-\pi_b(j_i))$. Let
$l^*=(m_{i_1},...,m_{i_t})$. Clearly then, $\pi_e\in \Pi_{L_W,
n,l^*}^*(w)$. And it follows from the above argument that, $\pi_b
\in \Pi_{L_T, k,l^*}^*(w)$. By hypothesis, $l^* \in
\{l^1,l^2,...,l^m\}$ and hence, $\pi_e\in
\bigcup_{l\in\{l^1,...,l^m\}}\Pi_{L_W, n,l}^*(w)$.

Now the map $$\phi_B^{-1}(\pi_b)\rightarrow
\bigcup_{l\in\{l^1,...,l^m\}}\Pi_{L_W, n,l}^*(w) : {\pi\to \pi_e}$$
is clearly injective. To show that it is surjective, it is easy to
see that if $\pi_b\in \Pi_{L_T, k,l}^*(w)$, then for any $\pi_e \in
\Pi_{L_W, n,l}^*(w)$, the circuit determined by $(\pi_b,\pi_e)$ is
indeed in $\Pi_{L, nk}^*(w)$. This shows $\#\phi_B^{-1}(\pi_b)=\#
\bigcup_{l\in\{l^1,...,l^m\}}\Pi_{L_W, n,l}^*(w)$.

(ii) Proof of this part is similar to that of the previous one.
Let $\pi \in \phi_A^{-1}(\pi_e)$. Then we claim that
$\pi_b=\phi_B(\pi) \in \bigcup_{l\in\{l^1,...,l^m\}}\Pi_{L_T,
k,l}^*(w)$. To prove this, let $i_1,i_2,..,i_t$ be all the
non-zero generating vertices for $w$. Since $\pi_b\in \Pi_{L_T,
k}^*(w)$, then, $(\pi_b(i-1)-\pi_b(i))=\pm
(\pi_b(j_i-1)-\pi_b(j_i))$. For $i\in \{i_1,i_2,..,i_t\} $, define
\begin{equation*}
m_i=\begin{cases}1&\quad\text{if}\quad (\pi_b(i-1)-\pi_b(i))=(\pi_b(j_i-1)-\pi_b(j_i))\\
-1&\quad\text{if}\quad(\pi_b(i-1)-\pi_b(i))=-(\pi_b(j_i)-\pi_b(j_i-1)).
\end{cases}
\end{equation*}
As $\pi \in \Pi_{L, nk}^*(w)$ and $\phi_B(\pi)=\pi_b,
\phi_A(\pi)=\pi_e$, $m_i=1 \Rightarrow
(\pi_b(i-1),\pi_b(i))=(\pi_b(j_i-1),\pi_b(j_i))$, and $m_i=-1
\Rightarrow (\pi_b(i-1),\pi_b(i))=(\pi_b(j_i),\pi_b(j_i-1))$. Let
$l^*=(m_{i_1},...,m_{i_t})$. It follows from the above argument
that, $\pi_e \in \Pi_{L_W, n,l^*}^*(w)$. Now, $l^* \in
\{l^1,l^2,...,l^m\}$ by hypothesis, and hence, $\pi_b \in
\bigcup_{l\in\{l^1,...,l^m\}}\Pi_{L_T, k,l}^*(w)$ as in the
previous lemma. Now the map
$$\phi_A^{-1}(\pi_e)\mapsto
\bigcup_{l\in\{l^1,...,l^m\}}\Pi_{L_T, k,l}^*(w):\pi\mapsto
\pi_b$$ is as before a bijection, and hence the lemma is proved.
\end{proof}

Now we are ready to prove the theorem.\vskip5pt

In view of  Proposition \ref{prop:review}~\eqref{item1}, without
loss of generality, we assume the input sequences to be uniformly
bounded.
Since the blocks $A_{n,i}$ satisfy Assumption B, it is clear that
$B_k\star A_{n, i}$ satisfies Property B. Hence
Proposition~\ref{prop:review} ~\eqref{item3} implies that
it is enough to show that, for every $t>0$ and
for every $w \in \bw_{2t}(2)$,
$p(w)$
exists.\vskip5pt

(i) Let $k$ be fixed and let $n\rightarrow \infty$.  Let $w \in \bw_{2t}(2)$.
We need to show that $\lim_{n\rightarrow
\infty}\frac{1}{(nk)^{t+1}}\#\Pi_{L, nk}^*(w)$ exists. Note that,
\begin{equation}
\label{dec1} \#\Pi_{L, nk}^*(w)=\sum_{\pi_b\in \Pi_{L_T,
k}^*(w)}\#\phi_B^{-1}(\pi_b).
\end{equation}
Let $\pi_b\in \Pi_{L_T, k}^*(w)$ and let $l^1,...,l^m$ as in
Lemma~\ref{lemma:wigner} \eqref{toeplitz}. Hence,
\begin{equation}
\label{twdec}
\frac{\#\phi_B^{-1}(\pi_b)}{(nk)^{t+1}}=\frac{1}{k^{t+1}}\frac{\#\bigcup_{l\in\{
l^1,...,l^m\}}\Pi_{L_W, n,l}^*(w)}{n^{t+1}}.
\end{equation}

Hence, from Lemma~\ref{result2} (3)  and \eqref{twdec}, it is
clear that
\begin{eqnarray}
\lim_{n\rightarrow \infty}
\frac{\#\phi_B^{-1}(\pi_b)}{(nk)^{t+1}}=\left\{
\begin{array}{l}
\frac{1}{k^{t+1}}~\text{if $w$ is Catalan and } \pi_b\in \Pi_{L_T,
k,l_0}^*(w)\\ \\
0,~~\text{otherwise}.
\end{array}
\right.
\end{eqnarray}

Now from (\ref{dec1}), it is clear that,
\begin{eqnarray*}
p(w)= \lim_{n\rightarrow \infty}\frac{1}{(nk)^{t+1}}\# \Pi_{L,
nk}^*(w)= \left\{ \begin{array}{l} \frac{\#\Pi_{L_T,
k,l_0}^*(w)}{k^{t+1}}~\text {if $w$ Catalan}\\
\\
0~\text{otherwise}.
\end{array} \right.
\end{eqnarray*}

In particular, $p(w)$ exists for all pair-matched $w$ and hence
LSD exists. All the odd moments are 0 and the $(2t)-th$ moment of
the limiting distribution is given by
$$\beta_{2t}=\sum_{w \in \bw_{2t}(2)}
p(w)= \sum_{w \in \bw_{2t}(2)}
\frac{\#\Pi_{L_T, k,l_0}^*(w)}{k^{t+1}}. $$

 (ii) Now let $n$ be fixed and let $k\rightarrow \infty$.  Let $w\in \bw_{2t}(2)$.
We need to show that $\lim_{k\rightarrow
\infty}\frac{1}{(nk)^{t+1}}\#\Pi_{L, nk}^*(w)$ exists. Note that,
\begin{equation}
\label{dec2} \#\Pi_{L, nk}^*(w)=\sum_{\pi_e\in \Pi_{L_W, n}^*(w)}\#
\phi_A^{-1}(\pi_e).
\end{equation}
Let $\pi_e\in \Pi_{L_W, n}^*(w)$ and let $l^1,...,l^m$ as in Lemma
\ref{lemma:wigner}. Hence,
\begin{equation}
\label{twdec2}
\frac{\#\phi_A^{-1}(\pi_e)}{(nk)^{t+1}}=\frac{1}{n^{t+1}}\frac{\#\bigcup_{l\in\{
l^1,...,l^m\}}\Pi_{L_T, k,l}^*(w)}{k^{t+1}}.
\end{equation}

Hence, from Lemma~\ref{result2} and \eqref{twdec2}, it is clear that
$$p(w)=\lim_{k\rightarrow
\infty}\frac{1}{(nk)^{t+1}}\# \Pi_{L, nk}^*(w)=\frac{\#\Pi_{L_W,
n,l_0}^*(w)}{n^{t+1}} p_T(w).$$ Hence LSD of
$\frac{1}{\sqrt{nk}}B_k\star A_{n,i}$ exists if $n$ is fixed and
$k\rightarrow \infty$, is symmetric and is determined by its even
moments
$$\beta_{2t}=\sum_{w \in \bw_{2t}(2)}
\frac{\#\Pi_{L_W, n,l_0}^*(w)}{n^{t+1}} p_T(w).$$ (iii) From the
proofs of (i) and (ii) it is clear that, for any fixed $n$ and $k$,
$$(\#\Pi_{L_W, n,l_0}^*(w)) (\#\Pi_{L_T, k,l_0}^*(w))\leq \#\Pi_{L, nk}^*(w)\leq
\#(\Pi_{L_W, n}^*(w)) (\#\Pi_{L_T, k}^*(w)). $$ Now, we know that
$p_T(w)=1$ if $w$ is a Catalan word. Hence, using
Lemma~\ref{result2},
$$\lim_{n,k \rightarrow
\infty}\frac{\#\Pi_{L_W, n,l_0}^*(w)}{n^{t+1}}\frac{\#\Pi_{L_T,
k,l_0}^*(w)}{k^{t+1}}= \lim_{n,k \rightarrow
\infty}\frac{\#\Pi_{L_W, n}^*(w)}{n^{t+1}}\frac{\#\Pi_{L_T,
k}^*(w)}{k^{t+1}}=1~\text{or}~0$$ according as $w$ is Catalan or
non-Catalan.

 Sandwiching we get $p(w)=\lim_{n,k\rightarrow
\infty}\frac{\#\Pi_{L, nk}^*(w)}{(nk)^{t+1}}=1$ or $0$ according as
$w$ is Catalan or non-Catalan. As a consequnce, the LSD  exists and
is semicircular
and the
proof is complete.
%
%
%

\subsection{Proof of Theorem \ref{theorem:toeplitzblock}}
Let $w \in \bw_{2t}(2)$.
Let $\pi \in
\Pi_{L, nk}^*(w)$. Let $\pi_b$ be the block address of $\pi$. Then
$\pi_b \in \Pi_{L_T, k}^*(w)$ where $\Pi_{L_T, k}^*(w)$ is as in
the previous case. As before let us denote by $\phi_B$ , the block
address function $\Pi_{L, nk}^*(w)\mapsto \Pi_{L_T, k}^*(w): \pi
\mapsto \pi_b$.

Now let $\pi \in \Pi_{L, nk}^*(w)$, Let $\pi_e$ be the entry
address of $\pi$. It follows that
$$L(\pi(i-1),\pi(i))= L(\pi(j-1),\pi(j))\Rightarrow
L_{2,m_1}(\pi_e(i-1),\pi_e(i))= L_{2,m_2}(\pi_e(j-1),\pi_e(j))$$
where $m_1=L_1(\pi_b(i-1),\pi_b(i))$ and
$m_2=L_1(\pi_b(j-1),\pi_b(j))$. Now
$L_{2,m_1}(\pi_e(i-1),\pi_e(i))= L_{2,m_2}(\pi_e(j-1),\pi_b(j))$
implies $m_1=\pm m_2$ and hence
$(\pi_e(i-1)-\pi_e(i))=\pm(\pi_e(j-1)-\pi_e(j))$ and hence
$\pi_e\in \Pi_{L_T, n}^*(w)$.
As before let us denote by $\phi_A$ , the entry address function
$\Pi_{L, nk}^*(w)\mapsto \Pi_{L_T, n}^*(w): \pi \mapsto \pi_e$.

Analogous to Lemma \ref{lemma:wigner} we have the following lemma
whose proof is omitted.
\begin{lemma}\label{lemma:toeplitz}Let $w \in \bw_{2t}(2)$.
\begin{enumerate}
\item \label{tt1}
 Let $\pi_b\in \Pi_{L_T, k}^*(w)$ and
$l^1,l^2,...,l^m$ be all the distinct values of $l$ such that,
$\pi_b\in \Pi_{L_T, k,l}^*(w)$. Then,
$$\#\phi_B^{-1}(\pi_b)=\#\bigcup_{l\in\{l^1,...,l^m\}}\Pi_{L_T, n,l}^*(w).$$
\item \label{tt2}
Let  $\pi_e\in \Pi_{L_T, n}^*(w)$ and $l^1,l^2,...,l^m$ be all the
distinct values of $l$ such that, $\pi_e\in \Pi_{L_T, n,l}^*(w)$.
Then,
$$\#\phi_A^{-1}(\pi_e)=\#\bigcup_{l\in\{l^1,...,l^m\}}\Pi_{L_T, k,l}^*(w).$$
\end{enumerate}
\end{lemma}

The proof of this theorem is along the same lines as the proof of
the previous theorem.

\noindent (i) Let $k$ be fixed and let $n\rightarrow \infty$.  Let
$w \in \bw_{2t}(2)$.
We need to show that $\lim_{n\rightarrow
\infty}\frac{1}{(nk)^{t+1}}\#\Pi_{L, nk}^*(w)$ exists. Note that,
\begin{equation}
\label{dec3} \#\Pi_{L, nk}^*(w)=\sum_{\pi_b\in \Pi_{L_T,
k}^*(w)}\#\phi_B^{-1}(\pi_b).
\end{equation}
Let $\pi_b\in \Pi_{L_T, k}^*(w)$ and let $l^1,...,l^m$ as in
Lemma~\ref{lemma:toeplitz} \eqref{tt1}. Hence,
\begin{equation}
\label{ttdec1}
\frac{\#\phi_B^{-1}(\pi_b)}{(nk)^{t+1}}=\frac{1}{k^{t+1}}\frac{\#\bigcup_{l\in\{
l^1,...,l^m\}}\Pi_{L_T, n,l}^*(w)}{n^{t+1}}.
\end{equation}
Now using Lemma~\ref{result2},
\begin{eqnarray}
\lim_{n\rightarrow \infty}
\frac{\#\phi_B^{-1}(\pi_b)}{(nk)^{t+1}}&=&\frac{1}{k^{t+1}}p_T(w)~\text{if}
~\pi_b\in \Pi_{L_T, k,l_0}^*(w)\\
\lim_{n\rightarrow \infty}\frac{\#\phi_B^{-1}(\pi_b)}{(nk)^{t+1}}&=&
0,~~\text{otherwise}.
\end{eqnarray}
Now from (\ref{dec3}), it is clear that,
\begin{equation}
p(w)=\lim_{n\rightarrow \infty}\frac{1}{(nk)^{t+1}}\#\Pi_{L,
nk}^*(w)=\frac{\#\Pi_{L_T, k,l_0}^*(w)}{k^{t+1}} p_T(w).
\end{equation}
In particular, $p(w)$ exists for all pair-matched $w$ and hence
LSD exists. All the odd moments are 0 and the $2t-th$ moment of
the limiting distribution is given by
$$\beta_{2t}=\sum_{w \in \bw_{2t}(2)}
p(w)= \sum_{w \in \bw_{2t}(2)}
\frac{\#\Pi_{L_T, k,l_0}^*(w)}{k^{t+1}}p_T(w).
$$ \noindent (ii) Proof of (ii) is exactly similar to that of (i),
except that we use Lemma~\ref{lemma:toeplitz}\eqref{tt2} instead
of Lemma~\ref{lemma:toeplitz} \eqref{tt1}.\vskip5pt

\noindent (iii) It is clear that for any fixed $n$ and $k$,
$$(\#\Pi_{L_T, n,l_0}^*(w)) (\#\Pi_{L_T, k,l_0}^*(w))\leq \#\Pi_{L, nk}^*(w)\leq
(\#\Pi_{L_T, n}^*(w)) (\#\Pi_{L_T, k}^*(w)). $$ Also, using
Lemma~\ref{result2},
$$\lim_{n,k \rightarrow
\infty}\frac{\#\Pi_{L_T, n,l_0}^*(w)}{n^{t+1}}\frac{\#\Pi_{L_T,
k,l_0}^*(w)}{k^{t+1}}= \lim_{n,k \rightarrow
\infty}\frac{\#\Pi_{L_T, n}^*(w)}{n^{t+1}}\frac{\#\Pi_{L_T,
k}^*(w)}{k^{t+1}} =(p_T(w))^2$$ for every pair-matched  $w$ of
length $(2t)$.

Sandwiching, we get $p(w)=\lim_{n,k\rightarrow
\infty}\frac{\#\Pi_{L, nk}^*(w)}{(nk)^{t+1}}=(p_T(w))^2$ for every
$w \in \bw_{2t}(2)$.
Hence the LSD in this case exists
and is identified by the even moments
$$\beta_{2t}=\sum_{w \in \bw_{2t}(2)}
(p_T(w))^2.$$
\end{document}